\documentclass[12pt, reqno]{amsart}

\usepackage{amssymb,amsmath,amsthm,graphicx,enumitem,systeme, array,enumitem,commath,mathtools,multicol,url}
\usepackage{setspace,blindtext}
\usepackage{varioref}
\usepackage[colorlinks=true,linkcolor=blue]{hyperref}
\usepackage[capitalise]{cleveref}
\usepackage{blindtext}
\usepackage{fancyhdr}
\usepackage{pgfplots}
\usepackage{geometry}
\usepackage{mathabx}

\usetikzlibrary{calc,angles,quotes}

\usepackage{xcolor}
\usepackage{caption}
\allowdisplaybreaks 

\theoremstyle{plain}
\newtheorem{theorem}{Theorem}[section]
\newtheorem{lemma}[theorem]{Lemma}
\newtheorem{corollary}[theorem]{Corollary}

\theoremstyle{definition}
\newtheorem{definition}[theorem]{Definition}
\newtheorem{example}[theorem]{Example}
\newtheorem{remark}[theorem]{Remark}

\title[Frullani-type integrals involving cosines]{Convergence criteria for Frullani-type integrals involving differences of cosines}

\author{Atiratch Laoharenoo}
\address{Department of Mathematics and Computer Science, Kamnoetvidya Science Academy, Rayong 21210, Thailand}
\email{atiratch.l@kvis.ac.th}

\author{Chanatip Sujsuntinukul}
\address{Department of Mathematics, The University of Hong Kong, Pokfulam, Hong Kong}
\email{chanatip@connect.hku.hk}



\begin{document}

\maketitle

\begin{abstract}
For $p,q\in\mathbb{N}$ and $\alpha,\beta\in\mathbb{R}$, we investigate the family of improper integrals
 \[\int_0^\infty\frac{(\cos\alpha x-\cos\beta x)^p}{x^q}dx.\] 
We establish a complete classification of the parameter ranges $(p, q; \alpha, \beta)$ for which the integrals converge or diverge, and we derive explicit closed-form evaluations in all convergent cases. The analysis also reveals a family of combinatorial identities arising naturally from  coefficients in the trigonometric power expansions. As a further application of the same method, we study an analogous class of integrals involving powers of sine differences. This extends the work of Laoharenoo and Boonklurb in 2022.
\end{abstract}

\section{Introduction}

In the early 19th century, an Italian mathematician Frullani introduced an interesting integral which relates a ``difference under scaling'' to a logarithmic factor \cite{AR}. In its standard form, if $f:[0,\infty)\to\mathbb{R}$ is continuous, has a finite limit at infinity $f(\infty):=\lim_{x\to\infty}f(x)$, and $\alpha, \beta>0$, then (assuming convergence) one has
\[
\int_{0}^{\infty}\frac{f(\alpha x)-f(\beta x)}{x}dx=\bigl(f(\infty)-f(0)\bigr)\log\frac{\alpha}{\beta}.
\]
If $f(\infty)$ does not exist, but $\int_{\gamma}^{\infty}(f(x)/x)dx$ exists for some $\gamma>0$, then one has the alternative form
\[
\int_{0}^{\infty}\frac{f(\alpha x)-f(\beta x)}{x}dx=-f(0)\log\frac{\alpha}{\beta}.
\]

This identity has served as a prototype for many extensions in which one seeks either broader hypotheses guaranteeing convergence and evaluation, or explicit closed forms for particular choices of $f$. For instance, a notable direction replaces pointwise information about $f(x)$ as $x\to\infty$ by averaged information. In 1949, Ostrowski \cite{OS} introduced the mean value
\begin{align}\label{mean_rewrite}
M(f):=\lim_{x\to\infty}\frac1x\int_0^x f(t)\,dt,
\end{align}
and showed that, under suitable assumptions ensuring both existence and convergence, the same Frullani integral can be expressed in terms of $M(f)$ via
\[
\int_0^\infty \frac{f(\alpha x)-f(\beta x)}{x}\,dx = M(f)\log\frac{\alpha}{\beta}.
\]
These integrals of Frullani-type have, in fact, attracted the interest of many mathematicians (including Ramanujan and Hardy, among others) during the mid 20th century and in recent years. See, for example, \cite{JB}, \cite{AJP}, \cite{ARP}, \cite{BS}, \cite{RM}. One recurring theme is understanding how oscillatory structure (such as taking $f$ to be sine or cosine functions) and modified powers in the numerator or denominator affect both convergence and the possibility of closed-form evaluation.

More recently, Laoharenoo and Boonklurb \cite{AL}, in 2022, introduced a family of Frullani-type integrals built from trigonometric oscillations,
\begin{align}\label{main}
    \mathcal{I}_{p, q}(\alpha, \beta):=\int_0^\infty f_{p, q}[\alpha, \beta](x)\,dx,
\end{align}
where \[f_{p, q}[\alpha, \beta](x):=\frac{(\cos \alpha x-\cos \beta x)^p}{x^q}.\]
In this case, we allow $p,q\in\mathbb{N}$ and $\alpha, \beta\in\mathbb{R}$, $|\alpha|\ne|\beta|$ (Indeed, since the cosine function is even, it suffices to consider when $\alpha, \beta\ge 0$.).
The special case $p=q=2$ appears (in a different context) in Hardy’s work \cite{GHH}. In addition, not so long ago, a similar class of such improper integrals has also been studied using the Mellin transform \cite{maag}. Beyond these connections, the general family seems not to have been systematically analysed, and in particular it is not recorded in the large compendium \cite{II}.

These integrals appear to be somewhat challenging to study. First, they exhibit singular behaviour at both endpoints of integration: the integrand has a potential singularity at $x=0$, whilst the upper limit extends to infinity. Establishing convergence thus requires careful analysis of both singularities. Moreover, extracting explicit values is nontrivial, due to the complexity of the oscillatory terms involved.
In \cite{AL}, they were able to compute the exact values of \eqref{main} for the cases where $p\ge q$ and $q$ is even. However, the remaining cases are yet to be investigated.

Based on the discussion above, this paper provides a complete analysis of the family of integrals in \eqref{main}. After collecting the necessary preliminaries in Section~\ref{sec2}, we first give a full classification of convergence in Section~\ref{sec3}. Owing to the atypical singular behaviour of \(f_{p,q}[\alpha,\beta]\), the treatment of the case \(q=1\) differs somewhat from that of \(q\ge 2\). Consequently, in Section~\ref{sec4}, we derive explicit closed-form values for \(\mathcal{I}_{p,q}(\alpha,\beta)\) in all convergent cases. In Section~\ref{sec5}, we will look at the special case when $\alpha=0$. Interestingly, these values turn out to exhibit some interesting combinatorial patterns. Finally, in Section \ref{sec6}, building on the same theory, we also roughly study the convergence classification of the family similar to \eqref{main}, but the cosine is replaced by the sine.

\section{Preliminaries}\label{sec2}



In this section, we collect several lemmas that will be important to our analysis later on. The first two lemmas can be found in the literature.


\begin{lemma}\cite{II}\label{basic1}
    Let $\alpha, \beta, \gamma\in\mathbb{R}$ such that  $\alpha, \beta\ne 0$. Then we have
    \[\int_0^{\infty}\frac{\cos\alpha x-\cos\beta x}{x}dx=\log\left|\frac{\beta}{\alpha}\right|\quad\text{and}\quad\int_0^{\infty}\frac{\sin \gamma x}{x}dx=\begin{cases}
        \displaystyle\frac{\pi}{2}\text{sgn}(\gamma) &\text{if } \gamma\ne 0,\\
        0 &\text{if } \gamma=0.\end{cases}\]
        In this case, for  $t\in\mathbb{R}$, we define the signum function by
    \[\text{sgn}(t):=\begin{cases}
        1 &\text{if } t>0,\\
        0 &\text{if } t=0,\\
        -1 &\text{if } t<0.
    \end{cases}\]
\end{lemma}


\begin{lemma}[\cite{AL}, Lemma 4.2]\label{l4}
		Let $f : \mathbb{R}\to\mathbb{R}$ be a smooth function for which $x=0$ is a zero of order $k$ for some $k\geq 2$. Suppose there exists an integer $1\le \ell\leq k$ such that $\int_0^\infty\frac{f^{(\ell-1)}(x)}{x}dx$ exists, and $f^{(m)}$ is bounded on $\mathbb{R}$ for $0 \leq m \leq \ell-2$. Then
		\[\int_0^\infty\frac{f(x)}{x^{\ell}}dx=\frac{1}{(\ell-1)!}\int_0^\infty\frac{f^{(\ell-1)}(x)}{x}dx.\]
\end{lemma}

The next three lemmas can be shown easily.

\begin{lemma}\label{basiclem}
    Let $\alpha\in[0, 1]$ and $x\in[0, \pi/2]$. Then we have
    \begin{align}\label{inequa}
        x\sin x+2\cos x\le \alpha x\sin\alpha x+2\cos\alpha x.
    \end{align}
\end{lemma}
\begin{proof}
First we record the elementary inequality\footnote{This can be proven using the mean value theorem.} $x\cos x\le\sin x$ for $x\in[0, \pi/2]$. Now fix $x\in[0, \pi/2]$ and let \[g_x(\alpha):=2\cos\alpha x+\alpha x\sin\alpha x\] for $\alpha\in(0, 1]$. Then
$g_x'(\alpha)=x(\alpha x\cos \alpha x-\sin\alpha x)\le 0$,
since $\alpha x\in[0, \pi/2]$, showing $g_x$ is decreasing on $\alpha\in(0, 1]$. Hence \eqref{inequa} holds for $\alpha\in(0, 1]$. For $\alpha=0$, we just let $\alpha\to 0^+$ and apply continuity.
\end{proof}

\begin{lemma}\label{lem.dec}  Let $0\le \alpha\le 1$ and $p, q\in\mathbb{N}$ with $q>2p>0$. Then $f_{p, q}[\alpha, 1](x)$ is positive and decreasing on $(0, \pi/2]$.\end{lemma}
\begin{proof}
    First we see that the function
    \[g(x):=\frac{\cos\alpha x-\cos x}{x^2}\]
    is decreasing on $(0, \pi/2]$ due to Lemma \ref{basiclem} and
    \[g'(x)=\frac{1}{x^3}(x\sin x+2\cos x-\alpha x\sin \alpha x-2\cos\alpha x)\le 0.\]
    Since $g(x)\ge 0$ for $x\in(0, \pi/2]$, thus implies $g^p$ is decreasing on $(0, \pi/2]$. Next as the function $1/x^{q-2p}$ is positive and decreasing on $(0, \pi/2]$, the product
    \[\frac{g^p(x)}{x^{q-2p}}=\frac{(\cos\alpha x-\cos x)^p}{x^q}\]
    is positive and decreasing on $(0, \pi/2]$ as desired.
\end{proof}

\begin{remark}
    Although we primarily care about $p, q\in\mathbb{N}$,  the result in fact holds for  $p, q\in\mathbb{R}$.
\end{remark}

\begin{lemma}\label{boundary0}
    Let $\alpha, \beta\in\mathbb{R}$ such that $|\alpha|\ne|\beta|$. If $p, q\in\mathbb{N}$ with $q\le 2p$, then we have\footnote{Here and in the sequel, we use the usual Landau's asymptotic notation $X=O(Y)$ to denote the bound $|X|\le CY$ for some absolute constant $C$.} $f_{p, q}[\alpha, \beta](x)=O(1)$ as $x\to 0$.
\end{lemma}
\begin{proof}
This is obvious by considering the Taylor expansion around $x=0$ of
\begin{align*}
    (\cos \alpha x-\cos\beta x)^p=\left(\sum_{n=1}^{\infty}(-1)^{n+1}\frac{(\beta^{2n}-\alpha^{2n})}{(2n)!}x^{2n}\right)^p=\left(\frac{\beta^2-\alpha^2}{2}\right)^px^{2p}+O(x^{2p+2}).
\end{align*}
Hence, as $x\to 0$, $f_{p, q}[\alpha, \beta]\to((\beta^2-\alpha^2)/2)^p$ if $q=2p$ and $f_{p, q}[\alpha, \beta]\to0$ if $q<2p$.
\end{proof}

\section{Convergence Classification}\label{sec3}

We first classify the convergence of $\mathcal{I}_{p, q}(\alpha, \beta)$ in the case $q=1$, which appears to be the most delicate. Our result depends on the parity of $p$.

\subsection{Case $q=1$, even $p$}




\begin{theorem}\label{thm3.1}
Let $\alpha, \beta\in\mathbb{R}$ with $|\alpha|\ne|\beta|$. If $p\in\mathbb{N}$ is even, then $\mathcal{I}_{p, 1}(\alpha, \beta)$ diverges.
\end{theorem}
\begin{proof}
 We first note that $\mathcal{I}_{p, 1}(\alpha, \beta)=\mathcal{I}_{p, 1}(\alpha/\beta, 1)$ by the substitution rule. Hence it suffices to consider the case $\beta=1$.
 
    By the binomial expansion theorem together with several trigonometric formulas (for more details, consult Subsubsection \ref{sec3.2.2} or Subsection \ref{sec4.1}), we can write
    \begin{equation}\label{evenp}
        (\cos \alpha x-\cos x)^p=\sum_{k=0}^ta_k\cos d_kx,
    \end{equation}
    for some $t\in\mathbb{N}$. Here we assume $a_0\in\mathbb{R}$, $d_0=0$, $a_k, d_k\in\mathbb{R}\setminus\{0\}$ and $d_k$'s are distinct, where $k\in\{1, \ldots, t\}$. We first claim that $a_0\ne 0$. Otherwise, for $R>0$, integrating both sides of \eqref{evenp} gives
    \begin{align}\label{evenp_2}
        \int_0^{R}(\cos\alpha x-\cos x)^pdx=\sum_{k=1}^ta_k\int_0^R\cos d_kxdx=\sum_{k=1}^t\frac{a_k\sin d_kR}{d_k}.
    \end{align}
    We see that the right-hand side of \eqref{evenp_2} is bounded. Now we claim that the left-hand side of \eqref{evenp_2} tends to $\infty$ as $R\to\infty$, which will lead to a contradiction. Indeed, if $\alpha\in\mathbb{Q}$, then the function $(\cos\alpha x-\cos x)^p$ is positive a.e. and periodic. Now, suppose $\alpha\not\in\mathbb{Q}$. By Weyl's criterion, we know that the sequence
    $\{\alpha(2k+1)\pi \pmod{2\pi} : k\in\mathbb{N}\}$
    is dense in $[0, 2\pi]$. In particular, there exist infinitely many $k\in\mathbb{N}$ such that
    \[|\alpha(2k+1)\pi-2m_k\pi|<\frac{1}{100} \quad\text{for some } m_k\in\mathbb{Z}.\]
    Now define $x_k=(2k+1)\pi$ for each $k$. Then we have $\cos x_k=-1$ and 
    \[\cos \alpha x_k=\cos(\alpha(2k+1)\pi)=\cos(2m_k\pi+\varepsilon_k)=\cos \varepsilon_k\]
    with $|\varepsilon_k|<1/100$. Hence $\cos \alpha x_k>\cos (1/100)=: c\in(1/2, 1)$. Therefore, for all $k$, $\cos \alpha x_k-\cos x_k>c+1$. Now, by continuity, we can find  $0<\delta<1/100$ (independent of $k$) such that 
    \[\cos\alpha x-\cos x>c+\frac{1}{2}\quad \text{for all } (x_k-\delta, x_k+\delta).\]
    Since $p$ is even, we hence yield
    \begin{align*}
        \int_0^{\infty}(\cos\alpha x-\cos x)^pdx\ge \sum_{k=1}^{\infty}\int_{x_k-\delta}^{x_k+\delta}(\cos\alpha x-\cos x)^pdx \ge \sum_{k=1}^{\infty}2\delta \left(c+\frac{1}{2}\right)^p=\infty.
    \end{align*}
    This implies the divergence of the integral.
    

 Next, let $S_k:=\sum_{i=0}^ka_i$. By plugging in $x=0$ in \eqref{evenp}, we note that $S_t= 0$. Then we have the Abel's summation formula\footnote{Please note that this technique will be very useful from time to time.}
        \begin{align*}
            (\cos\alpha x-\cos  x)^p &=a_0+\sum_{k=1}^t(S_k-S_{k-1})\cos d_kx\\
            &=a_0+\sum_{k=1}^tS_k\cos d_kx-\sum_{k=1}^tS_{k-1}\cos d_kx\\
            &=a_0(\cos d_0x-\cos d_1x)+\sum_{k=1}^{t-1}S_k(\cos d_kx-\cos d_{k+1}x)\\
            &=\sum_{k=0}^{t-1}S_k(\cos d_kx-\cos d_{k+1}x).
        \end{align*}
    Hence we can write
    \begin{align*}
        \mathcal{I}_{p, 1}(\alpha, 1)=\sum_{k=0}^{t-1}S_k\mathcal{I}_{1, 1}(d_k, d_{k+1}),
    \end{align*}
    which is divergent since $\mathcal{I}_{1, 1}(d_k, d_{k+1})=\log|d_{k+1}|-\log|d_k|$ for $k\in\{1, \ldots, t-1\}$ and $\mathcal{I}_{1, 1}(0, d_1)=\int_0^{\infty}((1-\cos d_1x)/x)dx$ diverges.
\end{proof}

\subsection{Case $q=1$, odd $p$}

\subsubsection{The set $\mathcal{F}_p$}

Now we turn our attention to when $q=1$ and $p$ is odd. This case is somewhat more challenging to discuss. Before that, we construct several notations below.

\begin{definition}
    Let $b\in\mathbb{N}$ be odd and $a\in\{0, 1, \ldots, b-1\}$. Denote
    \[\mathcal{R}_{a, b}:=\left\{a-2k : k\in\mathbb{Z} \text{ and } m_{a, b}\le k\le\left\lfloor\frac{a}{2}\right\rfloor\right\},\]
    where $m_{a, b}:=\max\{0, a-(b-1)/2\}$.
    Define
    \begin{align*}
        \mathcal{F}_{a, b}:=\left\{\frac{r}{s}\in[0, 1) : r\in \mathcal{R}_{a, b}\text{ and } s\in\{r+1, r+3, \ldots, b-a\}\right\}.
    \end{align*}
\end{definition}

\noindent
As an illustration, if $b=5$, then
\[\mathcal{F}_{0,5}=\{0\}=\mathcal{F}_{4,5},\quad \mathcal{F}_{1,5}=\left\{\frac{1}{2},\frac{1}{4}\right\},\quad\mathcal{F}_{2,5}=\left\{0,\frac{2}{3}\right\},\quad\mathcal{F}_{3,5}=\left\{\frac{1}{2}\right\}.\]

\begin{remark}
    In general, we have $\mathcal{F}_{0, p}=\mathcal{F}_{p-1, p}=\{0\}$ for every positive odd $p$.
\end{remark}

Let us show some basic properties of the set $\mathcal{F}_{a, b}$.

\begin{lemma}\label{lemmaofF}
    Let $b\in\mathbb{N}$ be odd.
    \begin{enumerate}
        \item[(a)]\label{l3.4a} Let $u, v\in\{0, 1, \ldots, b-1\}$. Then $\mathcal{F}_{u, b}\cap\mathcal{F}_{v, b}\ne\emptyset$ if and only if $u$ and $v$ have the same parity.
        \item[(b)] $\mathcal{F}_{a, b}\subseteq \mathcal{F}_{a-2, b}$ for every $a\in\{(b+1)/2, \ldots, b-1\}$.
    \end{enumerate}
\end{lemma}
\begin{proof}
    Part (a). ($\Rightarrow$) We prove by contrapositive. Suppose $u$ and $v$ have different parity, says $u$ is odd while $v$ is even. Assume to the contrary that there exists $y\in\mathcal{F}_{u, b}\cap\mathcal{F}_{v, b}$. Then we can write $y=c/d=e/f$ where $c, f$ are odd while $d, e$ are even. However this is absurd since $cf=de$.\\
    ($\Leftarrow$) There are two cases to consider. First if both $u$ and $v$ are even, then one can easily see that $0\in\mathcal{R}_{u, b}\cap\mathcal{R}_{v, b}$ because $2\lfloor u/2\rfloor=u$, $2\lfloor v/2\rfloor=v$, and so $0\in\mathcal{F}_{u, b}\cap\mathcal{F}_{v, b}$.  On the other hand, if both $u$ and $v$ are odd, then one can check that $1/2\in\mathcal{F}_{u, b}\cap\mathcal{F}_{v, b}$.

    Part (b). It suffices to show that $\mathcal{R}_{a, b}\subseteq \mathcal{R}_{a-2, b}$. We observe that
    \begin{align*}
    m_{a, b}=\begin{cases}
        0 &\text{if } a\in\{0, 1, \ldots, (b-1)/2\},\\
        a-(b-1)/2\ge 1 &\text{if } a\in\{(b+1)/2, \ldots, b-1\}.
    \end{cases}
    \end{align*}
    If $a=(b+1)/2$, then $m_{a, b}-m_{a-2, b}=1=\lfloor a/2\rfloor-\lfloor (a-2)/2\rfloor$, and so $\mathcal{R}_{a, b}=\mathcal{R}_{a-2, b}$. Otherwise, we have
    \begin{align*}
        \mathcal{R}_{a, b}&=\left\{-a+(b-1), -a+(b-1)-2, \ldots, a-2\left\lfloor \frac{a}{2}\right\rfloor\right\},
    \end{align*}
    and $\mathcal{R}_{a-2, b}=\mathcal{R}_{a, b}\cup\{-a+(b-1)+2\}$.
    This clearly indicates that $\mathcal{R}_{a, b}\subseteq \mathcal{R}_{a-2, b}$.
\end{proof}

Due to Lemma \ref{lemmaofF}(b), it then makes sense to define the set\footnote{To the best of our current knowledge, the set $\mathcal{F}_p$ does not appear in the existing literature. Nevertheless, it exhibits intriguing combinatorial features that suggest a promising direction for future research.}
\[\mathcal{F}_{b}:=\bigcup_{a\in\{0, 1, \ldots, (b-1)/2\}}\mathcal{F}_{a, b}.\]
For example, $\mathcal{F}_5=\{0, 1/2, 1/4, 2/3\}$. It is easy to verify that $\max\mathcal{F}_b=(b-1)/(b+1)$. Now, the criterion we found for our odd $p$ situation is the following.

\begin{theorem}\label{p=1,oddp}
    Let $\alpha, \beta\in\mathbb{R}$ with $|\alpha|\ne|\beta|$, and let $p\in\mathbb{N}$ be odd.  Then
		$\mathcal{I}_{p,1}(\alpha,\beta)$ diverges if and only if either $\lvert\alpha/\beta\rvert\in\mathcal{F}_p$ or $\lvert\beta/\alpha\rvert\in\mathcal{F}_p$.
\end{theorem}

We need to introduce some more technical knowledge before proving this.


\subsubsection{An expansion of $\mathcal{J}$}\label{sec3.2.2}

Let $\alpha, \beta\in\mathbb{R}$. A key ingredient for our analysis is expanding the expression
\[\mathcal{J}(x):=(\cos \alpha x-\cos \beta x)^p\]
as a sum of unpowered cosines. There will be two distinct expansions of $\mathcal{J}$ in this paper. The first one we will be presenting here is ideal for classifying convergence, while the other one (which will be introduced in the next section) facilitates explicit evaluation once convergence is established.

\begin{lemma}\cite{II}
    Let $n\in\mathbb{N}\cup\{0\}$ and $x, y\in\mathbb{R}$. Then we have
    \[\cos^nx=\begin{cases}
				\displaystyle\frac{1}{2^{n-1}}\sum_{k=0}^{\frac{n-1}{2}}{n\choose k}\cos((n-2k)x)&\text{if }n\text{ is odd,}\\\\
				\displaystyle\frac{1}{2^n}{n\choose n/2}+\frac{1}{2^{n-1}}\sum_{k=0}^{\frac{n-2}{2}}{n\choose k}\cos((n-2k)x)&\text{if }n\text{ is even.}
			\end{cases}\]
\end{lemma}

Now suppose $n\in\mathbb{N}$ is even while $n'\in\mathbb{N}$ is odd. By the product-to-sum formula, we have
\begin{align*}
     \cos^nx\cos^{n'}y&=\frac{1}{2^{n+n'-1}}\sum_{m=0}^{\frac{n-2}{2}}\sum_{\ell=0}^{\frac{n'-1}{2}}\binom{n}{m}\binom{n'}{\ell}(\cos(\omega_{n,n',s, \ell}^+(x, y))+\cos(\omega_{n,n',s, \ell}^-(x, y)))\\
        &\quad+\frac{1}{2^{n+n'-1}}\sum_{\ell=0}^{\frac{n'-1}{2}}\binom{n}{n/2}\binom{n'}{\ell}\cos(\omega_{n',\ell}(y)).
\end{align*}
Here we make the notations
\begin{align*}
    \omega_{n,n',m, \ell}^+(x, y):=(n-2m)x+(n'-2\ell)y,\quad \omega_{n,n',m,\ell}^-(x, y):=(n-2m)x-(n'-2\ell)y.
\end{align*}
In particular, when $x=0$, we have $\omega_{n, n', m, \ell}^{\pm}(0, y)=(n'-2\ell)y$, and when $y=0$, 
we have $\omega_{n,n',m, \ell}^{\pm}(x, 0)=(n-2m)x$. For brevity, we write these as 
$\omega_{n',\ell}(y)$ and $\omega_{n,m}(x)$ respectively.

Finally, fix $p, q, \alpha, \beta$, and let $p\geq3$ be odd. Using this and the binomial expansion theorem, we can write 
\begin{equation}\label{decomJ}
    \mathcal{J}(x)=\mathcal{J}_1(x)-\mathcal{J}_2(x),
\end{equation} where
\begin{align*}
    \mathcal{J}_1 &:=\sum_{\substack{k \text{ even}\\ 0\le k\le p-1}}\binom{p}{k}\cos^{p-k}\alpha x\cos^k \beta x\\
    &=\frac{1}{2^{p-1}}\sum_{\substack{k\text{ even}\\2\leq k\leq p-1}}\sum_{m=0}^{\frac{k-2}{2}}\sum_{\ell=0}^{\frac{p-k-1}{2}}\xi_{k, \ell,m}(\cos(\omega^+_{k,p-k,m,\ell}( \beta x,\alpha x))+\cos(\omega^-_{k,p-k,m,\ell}(\beta x,\alpha x)))\\
    &\quad+\sum_{\substack{k\text{ even}\\0\leq k\leq p-1}}\frac{1}{2^{p-k-1}}\sum_{\ell=0}^{\frac{p-k-1}{2}}\xi_{k, \ell,k/2}{k\choose k/2}\cos(\omega_{p-k,\ell}(\alpha x))
\end{align*}
and
\begin{align*}
    \mathcal{J}_2 &:=\sum_{\substack{k \text{ odd}\\ 1\le k\le p}}\binom{p}{k}\cos^{p-k}\alpha x\cos^k \beta x\\
    &=\frac{1}{2^{p-1}}\sum_{\substack{k \text{ odd}\\ 1\leq k\leq p-2}}\sum_{m=0}^{\frac{p-k-2}{2}}\sum_{\ell=0}^{\frac{k-1}{2}}\xi_{k,m,\ell}(\cos(\omega^+_{p-k,k,m,\ell}(\alpha x, \beta x))+\cos(\omega^-_{p-k,k,m,\ell}(\alpha x, \beta x)))\\
    &\quad+\sum_{\substack{k \text{ odd}\\ 1\leq k\leq p}}\frac{1}{2^{k-1}}\sum_{\ell=0}^{\frac{k-1}{2}}\xi_{k,(p-k)/2,\ell}\cos(\omega_{k,\ell}(\beta x)).
\end{align*}
Here, for convenience, we use the notation
\[\xi_{k,a,b}={p\choose k}{p-k\choose a}{k\choose b}.\]

\subsubsection{Finishing the proof of Theorem \ref{p=1,oddp}}

Next, we need some additional lemmas before we prove Theorem \ref{p=1,oddp}.

\begin{lemma}\label{l3.7}
    Let $p\geq3$ be odd. For $k\in\{0, 1, \ldots, (p-1)/2\}$, define
    \begin{align*}
        A_k = \left\{0, 1, \ldots, \left\lfloor\frac{k}{2}\right\rfloor\right\}
        \quad\text{and}\quad
        B_k =\left\{0, 1, \ldots, \left\lfloor\dfrac{p-k-1}{2}\right\rfloor\right\}.
    \end{align*}
    Consequently, define 
    \[\mathcal{A}_{p, k}:=\left\{\frac{k-2a}{p-k-2b}\in[0,1) :  a\in A_k, b\in B_k\right\},\quad \mathcal{A}_p=\bigcup_{k\in\{0, \ldots, \frac{p-1}{2}\}}\mathcal{A}_{p, k}.\]
    Then we have $\mathcal{A}_p=\mathcal{F}_p$.
\end{lemma}
    \begin{proof}
    Note that $0\leq a\leq\lfloor k/2\rfloor$, and $0\leq b\leq(p-k-1)/2\leq(p-1)/2$ for $0\leq k\leq (p-1)/2$. Hence
    $1\leq p-k-2b\leq p$ {and}  $k-2a\geq0$.
    This implies $k-2a\in\mathcal{R}_{k,p}$. Now observe that \[p-k-2b=(k-2a)+(p+2(a-b-k)).\]
    It is clear that $p+2(a-b-k)$ is odd and
        \[0< p+2(a-b-k)\leq p+2\left(\frac{k}{2}-k\right)=p-2k+2a.\]
    Thus we obtain
    \[\frac{k-2a}{p-k-2b}=\frac{k-2a}{(k-2a)+(p+2(a-b-k))}\in\mathcal{F}_{k,p}\subset\mathcal{F}_p.\]

    Conversely, let $x\in\mathcal{F}_p$. Then there exist $k\in\{0,\ldots,(p-1)/2\}$, $a\in\{0,\ldots,\lfloor k/2\rfloor\}$ and $t\in\{1,3,\ldots,p-2k+2a\}$ such that
    \[x=\frac{k-2a}{k-2a+t}.\]
    It is easy to verify that $a\leq k/2$, and the following inequalities:
    \[k-2a+t=p-k-2\left(\frac{p-t}{2}+a-k\right),\quad 0\leq\frac{p-t}{2}+a-k\leq\frac{p-k-1}{2}.\]
    Since $t$ and $p$ are odd, so we can choose $b=a-k+(p-t)/2\in\mathbb{Z}$. Hence
    \[x=\frac{k-2a}{p-k-2b}\in\mathcal{A}_p\]
    as desired.
\end{proof}

\begin{lemma}\label{l3.8}
    Let $p\geq3$ be odd, $k\in\{0,\ldots,(p-1)/2\}$, and $a,b\in\mathbb{Z}$ such that $0\leq a\leq k/2$ and $0\leq b\leq (p-k-1)/2$, and 
    \[\frac{\alpha}{\beta}=\frac{k-2a}{p-k-2b}\in[0,1).\]
    \begin{enumerate}
        \item[(a)] If $k$ is even, then $\cos(\omega^-_{k,p-k,a,b}(\beta x,\alpha x))$ is constant, but $\cos(\omega^-_{p-k,k,b,a}(\alpha x,\beta x))$ is not, on $\mathbb{R}$.
        \item[(b)] If $k$ is odd, then $\cos(\omega^-_{p-k,k,b,a}(\alpha x,\beta x))$ is constant, but $\cos(\omega^-_{k,p-k,a,b}(\beta x,\alpha x))$ is not, on $\mathbb{R}$.
    \end{enumerate}
\end{lemma}
\begin{proof}
    First note that $(k-2a)/(p-k-2b)\in\mathcal{F}_{k, p}$.
    Then we observe that
    \[\omega^-_{k,p-k,a,b}(\beta x,\alpha x)=((k-2a)\beta-(p-k-2b)\alpha)x\]
    if $k$ is even, while
    \[\omega^-_{p-k,k,b,a}(\alpha x,\beta x)=((p-k-2b)\alpha-(k-2a)\beta)x\]
    if $k$ is odd. The statement $\alpha/\beta=(k-2a)/(p-k-2b)$ and Lemma \ref{l3.4a}(a) imply \[\cos(\omega^-_{k,p-k,a,b}(\beta x,\alpha x)) \text{ and } \cos(\omega^-_{p-k,k,b,a}(\alpha x,\beta x))\] cannot be constant on $\mathbb{R}$ simultaneously. In particular,
    $\cos(\omega^-_{k,p-k,a,b}(\beta x,\alpha x))=1$
    for $x\in\mathbb{R}$ and $\cos(\omega^+_{p-k,k,b,a}(\alpha x,\beta x))$ is not a constant on $\mathbb{R}$ if $k$ is even and vice versa. 
\end{proof}





Finally, we can prove the theorem we want. Let us first outline the idea (which is analogous to the even $p$ case). Firstly, we know that we can write
\[(\cos \alpha x-\cos  \beta x)^p=\sum_{k=0}^ta_k\cos d_kx\]
for some $t\in\mathbb{N}$. Here $a_0\in\mathbb{R}$, $a_k\in\mathbb{R}\setminus\{0\}$, $d_0=0$, and $d_k$ are distinct, where $k\in\{1, \ldots, t\}$.
Then we can define $S_k=\sum_{i=0}^ka_i$. Note that $S_t=0$ as usual. Then we can perform Abel's summation argument (See Theorem \ref{thm3.1} for reference). The problem then boils down to whether $a_0=0$ which is our goal of the upcoming investigation.

To explain intuitively, the set $\mathcal{F}_p$ we define characterises ratios $\alpha/\beta$ for which certain cosine frequencies cancel, producing a constant term in the expansion of $\mathcal{J}$.

\begin{proof}[Proof of Theorem \ref{p=1,oddp}]
    Without loss of generality, assume that $\beta\neq 0$. By the substitution rule, we have
    $\mathcal{I}_{p,1}(\alpha,\beta)=\mathcal{I}_{p,1}(\alpha/\beta,1)$
    for every $\alpha\in\mathbb{R}$. Moreover, since cosine is even,
    $\mathcal{I}_{p,1}(\alpha,\beta)=\mathcal{I}_{p,1}(\lvert\alpha\rvert,\lvert\beta\rvert)$. Hence it suffices to consider only the case $\alpha/\beta\geq 0$ and $\alpha\geq 0$. In addition, if $\alpha/\beta>1$, then we have
    $\mathcal{I}_{p,1}(\alpha/\beta,1)=-\mathcal{I}_{p,1}(\beta/\alpha,1)$.
    So we can focus only on the case $0\leq\alpha/\beta<1$.

($\Leftarrow$) First suppose $\alpha/\beta\in\mathcal F_p$. By Lemma \ref{l3.7}, there exist
$
k\in\{0,\ldots,(p-1)/2\}$, $m,\ell\in\mathbb Z$
with $0\le m\le k/2$ and $0\le\ell\le(p-k-1)/2$ such that
\[\frac{\alpha}{\beta}=\frac{k-2m}{p-k-2\ell}.\]

\textit{Case 1: $k$ even.}
If $k=0$, then $\alpha/\beta=0$. Expanding $(1-\cos x)^p$ gives
\[(1-\cos x)^p=a_0+\sum_{k=1}^t a_k\cos d_kx,\quad a_0\ne 0.\]
Assume $k\ge2$. If $m=k/2$, then again $\alpha/\beta=0$.  
If $0\le m\le(k-2)/2$, then by Lemma \ref{l3.8},
$\cos(\omega^-_{k,p-k,m,\ell}(\beta x,\alpha x))=1$
for all $x\in\mathbb R$, while
$\cos(\omega^-_{p-k,k,\ell,m}(\alpha x,\beta x))$
is not constant. Hence the trigonometric expansion $\mathcal{J}_1$ (only) contains a positive constant term.

\textit{Case 2: $k$ odd.}
In this case $0\le m\le(k-1)/2$ and $0\le\ell\le(p-k-2)/2$. Again by Lemma \ref{l3.8},
$\cos(\omega^-_{p-k,k,\ell,m}(\alpha x,\beta x))=1$
for all $x$, while the other cosine is not constant. Thus the expansion $\mathcal{J}_2$ contains a negative constant term.


($\Rightarrow$) Conversely, if $\alpha/\beta\notin\mathcal F_p$, then
\[
\frac{\alpha}{\beta}\neq\frac{k-2m}{p-k-2\ell}
\]
for every admissible $k,m,\ell$. Hence none of the frequencies vanish, so the expansion
\[
\left(\cos\left(\frac{\alpha}{\beta}x\right)-\cos x\right)^p
=
\sum_{k=1}^t a_k\cos d_kx
\]
contains no constant term by Lemma \ref{l3.7}.
\end{proof}

\subsection{Cases $q\ge 2$}

Now we consider a larger $q$. First we are able to establish a sufficient condition on $p, q$ so that $\mathcal{I}_{p, q}(\alpha, \beta)$ is divergent.

\begin{theorem}\label{divq>=2}
    Let $\alpha, \beta\in\mathbb{R}$ with $|\alpha|\ne|\beta|$. If $p, q\in\mathbb{N}$ such that $q\ge 2p+1$, then $\mathcal{I}_{p, q}(\alpha, \beta)$ diverges.
\end{theorem}
\begin{proof}
By the substitution rule and the fact that cosine is an even function, it suffices to consider only the case $\beta=1$ and $\alpha\ge 0$. Moreover, if $\alpha>1$, then we have the relation 
    \[\mathcal{I}_{p, q}(\alpha, 1)=(-1)^p\alpha^{q-1}\mathcal{I}_{p, q}(1/\alpha, 1).\]
    Hence we can focus only on the case $0\le \alpha<1$.
    
     We will show that $\int_0^1f$ is divergent. First $f$ is positive and decreasing on $(0, \pi/2]$ according to Lemma \ref{lem.dec}. Thus, for each $k\in\mathbb{N}$, we have
    \begin{align*}
        \int_{\frac{1}{k+1}}^{\frac{1}{k}}f_{p, q}[\alpha, 1](x)dx &\ge f_{p, q}[\alpha, 1]\left(\frac{1}{k}\right)\left(\frac{1}{k}-\frac{1}{k+1}\right)
        \\
        &=\left(\cos\left(\frac{\alpha}{k}\right)-\cos\left(\frac{1}{k}\right)\right)^p\frac{k^{q-1}}{k+1}>0.
    \end{align*}
    Hence we find
    \begin{align}\label{series}
        A_k:=\int_{\frac{1}{k+1}}^1f_{p, q}[\alpha, 1](x)dx\ge\sum_{n=1}^k\left(\cos\left(\frac{\alpha}{n}\right)-\cos\left(\frac{1}{n}\right)\right)^p\frac{n^{q-1}}{n+1}.
    \end{align}
    We will show that the sum in the right-hand side of \eqref{series} diverges as $k\to\infty$.
    
    {\it Case 1:} $q\ge 2p+2$. First by Taylor expansion around $x=0$, we see that
        \[\cos\left(\frac{\alpha}{n}\right)-\cos\left(\frac{1}{n}\right)=\frac{1-\alpha^2}{2n^2}+O\left(\frac{1}{n^4}\right).\]
    Thus we have
    \begin{align*}
        \lim_{n\to\infty}\left(\cos\left(\frac{\alpha}{n}\right)-\cos\left(\frac{1}{n}\right)\right)^p\frac{n^{q-1}}{n+1}=
        \begin{cases}
      \displaystyle\left(\frac{1-\alpha^2}{2}\right)^p  &\text{if } q=2p+2,\\\\
       \infty &\text{if } q>2p+2.
        \end{cases}
    \end{align*}
    This means $\lim_{k\to\infty}A_k$ is divergent.
    
    {\it Case 2:}  $q=2p+1$. 
    We apply the limit comparison test, with the divergent series  $\sum_{n\in\mathbb{N}}1/n^{2p-q+2}=\sum_{n\in\mathbb{N}}1/n$:
    \begin{align*}
        &\lim_{n\to\infty}n^{2p-q+2}\left(\cos\left(\frac{\alpha}{n}\right)-\cos\left(\frac{1}{n}\right)\right)^p\frac{n^{q-1}}{n+1} \\
        &\qquad=\lim_{n\to\infty}\left(n^2\left(\cos\left(\frac{\alpha}{n}\right)-\cos\left(\frac{1}{n}\right)\right)\right)^p\frac{n}{n+1}=\left(\frac{1-\alpha^2}{2}\right)^p>0.
    \end{align*}
    So $\lim_{k\to\infty} A_k$ diverges.
    Hence the proof is complete.
\end{proof}

This condition actually turns out to be necessary, as we shall see below.

\begin{theorem}\label{compar}
     Let $\alpha, \beta\in\mathbb{R}$ such that $|\alpha|\ne|\beta|$. If $p, q\in\mathbb{N}$ with $2\le q\le 2p$, then $\mathcal{I}_{p, q}(\alpha, \beta)$ exists.
\end{theorem}
\begin{proof}
    First we can decompose
    \[
\mathcal I_{p,q}(\alpha,\beta)
=\int_0^1 f_{p,q}[\alpha,\beta](x)\,dx+\int_1^\infty f_{p,q}[\alpha,\beta](x)\,dx.
\]
The first integral converges by Lemma~\ref{boundary0}. Now for $x\ge 1$, we see that
\[
|f_{p,q}[\alpha,\beta](x)|
=\frac{|\cos(\alpha x)-\cos(\beta x)|^p}{x^q}
\le \frac{2^p}{x^q},
\]
and since $q>1$, we have $\int_1^\infty x^{-q}\,dx<\infty$. Hence the second integral converges by the comparison test. 
\end{proof}


\section{Exact computations}\label{sec4}

We now try to calculate the precise values of $\mathcal{I}_{p, q}(\alpha, \beta)$ for the convergent cases.

\subsection{Another expansion of $\mathcal{J}$}\label{sec4.1}

To streamline the presentation, we introduce several pieces of notation. Fix $\alpha, \beta, p, q$. First denote
\begin{align*}
    \Omega_1&:=\{(k, \ell, m)\in(\mathbb{N}\cup\{0\})^3 : 0\le k\le p-1, 0\le \ell\le p-k-1, 0\le m\le k\},\\
    \Omega_2&:=\{(p, \ell, 0)\in(\mathbb{N}\cup\{0\})^3 : 0\le \ell\le p-1\},
\end{align*}
and $\Omega:=\Omega_1\cup\Omega_2$. Write
\begin{align*}
    \Lambda_{k, \ell, m}:=\begin{cases}
        \displaystyle\frac{(-1)^k}{2^{p-1}}\binom{p}{k}\binom{p-k-1} {\ell}\binom{k}{m}  &\text{if } k\in\{0, 1, \ldots, p-1\},\\\\
        \displaystyle\frac{(-1)^p}{2^{p-1}}\binom{p-1}{\ell} &\text{if }k=p,
    \end{cases}
\end{align*}
and
\begin{align*}
    \chi_{k, \ell, m}(x):=\begin{cases}
        (p-k-2\ell)\alpha x+(k-2m)\beta x &\text{if } k\in\{0, 1, \ldots, p-1\},\\
        (p-2\ell)\beta x &\text{if } k=p.
    \end{cases}
\end{align*}
From this, we additionally define  $\tilde{\Omega}:=\{(k, \ell, m)\in\Omega : \chi_{k, \ell, m}(1)\ne 0\}$. We also label the elements in $\tilde{\Omega}$ as $r_1, r_2, \ldots, r_N$ (the order does not matter).  Then for $1\le n\le N$, write
\begin{align*}
    \Xi_n:=\sum_{j=1}^n\Lambda_{r_j}(\chi_{r_j}(1))^{q-1}.
\end{align*}

As mentioned before, below we provide an alternative way of expanding $\mathcal{J}$.

\begin{lemma}[\cite{AL}, Corollary 3.2 and Lemma 4.1]\label{long}
    Let $x, y\in\mathbb{R}$. Let $n, n'\in\mathbb{N}\cup\{0\}$ such that $n\ge 1$. Then we have
     \[\cos^nx=\frac{1}{2^{n-1}}\sum_{k=0}^{n-1}\binom{n-1}{k}\cos((n-2k)x)\]
     and
    \[\cos^nx\cos^{n'}y=\frac{1}{2^{n+n'-1}}\sum_{\ell=0}^{n-1}\sum_{m=0}^{n'}\binom{n-1}{\ell}\binom{n'}{m}\cos((n-2\ell)x+(n'-2m)y).\]
    Consequently, combining this with the binomial expansion theorem, we yield
    \[\mathcal{J}(x)=\sum_{(k, \ell, m)\in\Omega}\Lambda_{k, \ell, m}\cos(\chi_{k, \ell, m}(x)).\]
\end{lemma}

\subsection{Case $q=1$}

As we know from the previous section, the operator  $\mathcal{I}_{p, q}(\alpha, \beta)$ is always divergent for $q=1$ when $p$ is even. Thus our attention would only be in the odd $p$ case.

\begin{theorem}\label{compuq=1}
    Let $\alpha, \beta\in\mathbb{R}$ with $|\alpha|\ne|\beta|$, and let $p\in\mathbb{N}$ be odd. Suppose both $|\alpha/ \beta|$ and $|\beta/\alpha|$ are not in $\mathcal{F}_p$. Then 
    \begin{align*}
        \mathcal{I}_{p, 1}(\alpha, \beta)=\sum_{n=1}^{N-1}\Xi_n\log\left|\frac{\chi_{r_{n+1}}(1)}{\chi_{r_{n}}(1)}\right|.
    \end{align*}
\end{theorem}
\begin{proof}
    As we have verified in Theorem \ref{p=1,oddp}, the assumption guarantees that the sum of all the terms which satisfy $(k,\ell ,m)\in\Omega\setminus\tilde{\Omega}$ equals $0$. We first write
    \begin{align*}
        \mathcal{I}_{p, 1}(\alpha, \beta)=\int_0^{\infty}\sum_{(k, \ell, m)\in\tilde{\Omega}}\Lambda_{k, \ell, m}\frac{\cos (\chi_{k, \ell, m}(x))}{x}dx.
    \end{align*}
    For the next step, we express $\mathcal{I}_{p, 1}(\alpha, \beta)$ in the Abel's summation style. Note that $\Xi_N=\mathcal{J}(0)=0$. Combining this idea with Lemma \ref{basic1}, we finally yield
    \begin{align*}
        \mathcal{I}_{p, 1}(\alpha, \beta) &=\sum_{n=1}^{N-1}\Xi_n\int_0^{\infty}\frac{\cos(\chi_{r_n}(x))-\cos(\chi_{r_{n+1}}(x))}{x}dx=\sum_{n=1}^{N-1}\Xi_n\log\left|\frac{\chi_{r_{n+1}}(1)}{\chi_{r_{n}}(1)}\right|,
    \end{align*}
    as desired.
\end{proof}

\subsection{Case $q\ge 2$}

The main tool for handling these cases is Lemma \ref{l4}, which transforms $\mathcal{I}_{p, q}(\alpha, \beta)$ into a form amenable to Lemma \ref{basic1}.

\begin{theorem}\label{mainthmcompute}
    Let $\alpha, \beta\in\mathbb{R}$ with $|\alpha|\ne|\beta|$. Let $p, q\in\mathbb{N}$ such that $2\le q\le 2p$.  
    \begin{enumerate}
        \item[(a)] If $q$ is even, then
        \begin{align*}
            \mathcal{I}_{p, q}(\alpha, \beta)=\frac{\pi(-1)^{\frac{q}{2}}}{2(q-1)!}\sum_{n=1}^N\Lambda_{r_n}|\chi_{r_n}(1)|^{q-1}.
        \end{align*}
        \item[(b)] If $q$ is odd, then
        \begin{align*}
            \mathcal{I}_{p, q}(\alpha, \beta)=\frac{(-1)^{\frac{q-1}{2}}}{(q-1)!}\sum_{n=1}^{N-1}\Xi_n\log\left|\frac{\chi_{r_{n+1}}(1)}{\chi_{r_{n}}(1)}\right|.
        \end{align*}
    \end{enumerate}
\end{theorem}
\begin{proof}
    First we can expand $(\cos \alpha x-\cos\beta x)^p$ as $\mathcal{J}(x)$ in Lemma \ref{long}. Note that $\mathcal{J}$ is smooth with a zero of order $2p$ at $x=0$.\\
    Part (a). As $q$ is even, we have
    \begin{align*}
        \mathcal{J}^{(q-1)}(x)=\sum_{(k, \ell, m)\in\tilde{\Omega}}(-1)^{\frac{q}{2}}\Lambda_{k, \ell, m}(\chi_{k, \ell, m}(1))^{q-1}\sin(\chi_{k, \ell, m}(x)).
    \end{align*}
    Note that we can ignore the vanishing terms, i.e., terms where $\chi_{k, \ell, m}(1)=0$, because they disappear during the differentiation. Thus using Lemma \ref{l4} and Lemma \ref{basic1} respectively, we yield
    \begin{align*}
       \mathcal{I}_{p, q}(\alpha, \beta)&=\frac{(-1)^{\frac{q}{2}}}{(q-1)!}\sum_{(k, \ell, m)\in\tilde{\Omega}}\Lambda_{k, \ell, m}(\chi_{k, \ell, m}(1))^{q-1}\int_0^{\infty}\frac{\sin(\chi_{k, \ell, m}(x))}{x}dx\\
       &=\frac{\pi(-1)^{\frac{q}{2}}}{2(q-1)!}\sum_{(k, \ell, m)\in\tilde{\Omega}}\Lambda_{k, \ell, m}|\chi_{k, \ell, m}(1)|^{q-1}.
    \end{align*}
    Part (b). The first step is similar to (a). First we have
    \begin{align*}
    \mathcal{I}_{p, q}(\alpha, \beta) &=\frac{(-1)^{\frac{q-1}{2}}}{(q-1)!}\int_0^{\infty}\sum_{(k, \ell, m)\in\tilde{\Omega}}\Lambda_{k, \ell, m}(\chi_{k, \ell, m}(1))^{q-1}\frac{\cos(\chi_{k, \ell, m}(x))}{x}dx
    \end{align*}
    due to Lemma \ref{l4}.
    Now we can try to write the expression in an Abel summation fashion.  Note that \[0=\mathcal{J}^{(q-1)}(0)=(-1)^{\frac{q-1}{2}}\sum_{(k, \ell, m)\in\tilde{\Omega}}\Lambda_{k, \ell, m}(\chi_{k, \ell, m}(1))^{q-1}=(-1)^{\frac{q-1}{2}}\Xi_{N}\] because $\mathcal{J}$ has a zero of order $2p$ and $q-1\le 2p-1$. In the end, applying Lemma \ref{basic1} gives us the desired result.
\end{proof}

\begin{remark}
While the formula may appear complex at first glance, the underlying algorithm is relatively straightforward, and one can readily implement code to evaluate this integral.
\end{remark}

\begin{example}
    We provide two beautiful examples of the operator $\mathcal{I}_{p ,q}(\alpha, \beta)$ under specific inputs as shown below.
    \begin{gather*}
        \mathcal{I}_{3, 2}(5, 2)=\int_0^{\infty}\frac{(\cos 5x-\cos 2x)^3}{x^2}dx=-\frac{3\pi}{4},\\
        \mathcal{I}_{3, 5}(3, 1)=\int_0^{\infty}\frac{(\cos 3x-\cos x)^3}{x^5}dx=\frac{1}{32}(2401\log 7-4590\log 3).
    \end{gather*}
\end{example}

\section{Special case $\alpha=0$ and combinatorial relations}\label{sec5}

We now focus on the special case $\alpha=0$. First, we record the following specialisation of Lemma \ref{long}. For $p\in\mathbb{N}$, we have
		\begin{align}\label{1-cosx}
        \begin{split}
						        (1-\cos \beta x)^p&=\sum_{k=1}^p\left((-1)^k\sum_{m=0}^{\lfloor (p-k)/2\rfloor}\frac{1}{2^{2m+k-1}}{p\choose 2m+k}{2m+k\choose m}\right)\cos k \beta x\\
            &\qquad+\sum_{m=0}^{\lfloor p/2\rfloor}\frac{1}{2^{2m}}{p\choose 2m}{2m\choose m}.
        \end{split}
		\end{align}
Let us, in another way, write
\begin{equation}\label{anotherwaytowrite}
    (1-\cos \beta x)^p=\sum_{k=0}^pa_k\cos k\beta x,
\end{equation}
where $a_k$ is the respective coefficient of $\cos k\beta x$ for $k\in\{0,1,\ldots,p\}$.
Indeed, by \eqref{1-cosx} and a well-known combinatorial argument, we see that
\begin{equation}\label{a0}
    a_0=\sum_{m=0}^{\lfloor p/2\rfloor}\frac{1}{2^{2m}}\binom{p}{2m}\binom{2m}{m}=\frac{1}{2^p}\binom{2p}{p}.
\end{equation}

Our next goal is to find simpler forms of the other $a_k$'s.  First let us extract some properties of these coefficients. A trivial observation (which we have seen before) is that substituting $x=0$ gives 
\begin{equation}\label{easy}
    \sum_{k=0}^pa_k=0.
\end{equation}
Perhaps we can deduce some more interesting consequences as follows.

\begin{lemma}\label{diffmanytimes}
Let $p\in\mathbb{N}$. If $n\in\mathbb{N}$ is even such that $2\le n\le 2p-2$, then
    \[\sum_{k=1}^pk^na_k=0.\]
\end{lemma}
\begin{proof}
    Taking the $n$th derivative on both sides of \eqref{anotherwaytowrite} with respect to $x$, we have
    \begin{align*}
        \frac{d^n}{dx^n}(1-\cos x)^p=(-1)^{n/2}\sum_{k=1}^pk^na_k\cos kx.
    \end{align*}
    As the function $(1-\cos x)^p$ has a zero of order $2p$ at $x=0$, we find
    \[0=\frac{d^n}{dx^n}(1-\cos x)^p\bigg|_{x=0}=(-1)^{n/2}\sum_{k=1}^pk^na_k\]
    as desired.
\end{proof}

From this, we are able to explicitly find the coefficients $a_k$'s as follows.

\begin{theorem}\label{rmk5.3}
    Let $a_k$'s be as defined in \eqref{anotherwaytowrite}. Then we have
     \[a_k=\begin{cases}
			-\dfrac{(2p)!}{2^p\displaystyle\prod_{2\le m\le p}(m^2-1)}&\text{if }k=1,\\
			\\
			(-1)^p\dfrac{(2p)!}{2^p\displaystyle\prod_{\substack{0\leq m\leq p\\m\neq k}}(k^2-m^2)}&\text{if }2\leq k\leq p.
		\end{cases}\]
\end{theorem}
\begin{proof}
    According to \eqref{easy} and Lemma \ref{diffmanytimes}, we obtain a system of $p$ linear equations
    \begin{align*}
        \begin{pmatrix}
            1 & 1 & \cdots & 1\\
            1 & 2^2 & \cdots & p^2\\
            \vdots & \vdots & \ddots & \vdots\\
            1 & 2^{2p-2} & \cdots & p^{2p-2}
        \end{pmatrix}
        \begin{pmatrix}
            a_1 \\ a_2 \\ \vdots \\ a_p
        \end{pmatrix}=\begin{pmatrix}
            -a_0 \\ 0 \\ \vdots \\ 0
        \end{pmatrix}.
    \end{align*}
    For convenience, we write the system as $P\mathbf{a}=\mathbf{b}$. We will solve this system using the Cramer rule. First note that the coefficient matrix $P$ appears to be the Vandermonde matrix. In this case, we know that
    \[\det P=\prod_{1\le n<m\le p}(m^2-n^2).\]
    Next we see that
    \begin{align*}
        \det\begin{pmatrix}
            -a_0 & 1 & \cdots & 1\\
            0 & 2^2 & \cdots & p^2\\
            \vdots & \vdots & \ddots & \vdots\\
            0 & 2^{2p-2} & \cdots & p^{2p-2}
        \end{pmatrix}
        &=-a_0(p!)^2\det\begin{pmatrix}
            1 & \cdots & 1\\
            2^2 & \cdots & p^2\\
            \vdots & \ddots & \vdots\\
            2^{2p-4} & \cdots & p^{2p-4}
        \end{pmatrix}\\
        &=-a_0(p!)^2\prod_{2\le n<m\le p}(m^2-n^2).
    \end{align*}
    Moreover, we have
    \begin{align*}
        \det\begin{pmatrix}
				1&1&\cdots&\overbrace{-a_0}^{\text{$k$th column}}&\cdots&1\\
				1&2^2&\cdots&0&\cdots&p^2\\
				1&2^4&\cdots&0&\cdots&p^4\\
				\vdots&\vdots&\ddots&\vdots&\cdots&\vdots\\
				1&2^{2p-2}&\cdots&0&\cdots&p^{2p-2}
			\end{pmatrix}=(-1)^{k}a_0\frac{(p!)^2}{k^2}\prod_{\substack{1\le n<m\le p\\ n, m\ne k}}(m^2-n^2).
    \end{align*}
    Then by using \eqref{a0} and performing further calculations, we are done.
\end{proof}

\begin{remark}
    \label{explicita_k}
    There is an alternative way of determining these $a_k$'s.  First, using the half-angle formula, we have
    \begin{align}\label{rmkeq1}
        (1-\cos \beta x)^p=2^p\sin^{2p}\frac{\beta x}{2}.
    \end{align}
    Then, due to Euler's formula and the binomial expansion theorem, we obtain
    \begin{align}\label{rmkeq2}
        \sin^{2p}\frac{\beta x}{2}=\left(\frac{e^{i\beta x/2}-e^{-i\beta x/2}}{2i}\right)^{2p}=\frac{(-1)^p}{2^{2p}}\sum_{m=0}^{2p}(-1)^m\binom{2p}{m}e^{i(p-m)\beta x}.
    \end{align}
    Hence combining \eqref{rmkeq1} and \eqref{rmkeq2} gives
    \[(1-\cos\beta x)^p=\frac{(-1)^p}{2^p}\sum_{m=0}^{2p}(-1)^m\binom{2p}{m}e^{i(p-m)\beta x}.\]
    Finally, since $e^{i\beta kx}+e^{-ik\beta x}=2\cos k\beta x$, pairing the terms $m$ and $2p-m$ produces
    \[(1-\cos \beta x)^p=\frac{1}{2^p}\binom{2p}{p}+\sum_{k=1}^p\frac{(-1)^k}{2^{p-1}}\binom{2p}{p-k}\cos k\beta x.\]
    Therefore we find
    \[a_0=\frac{1}{2^p}\binom{2p}{p}\quad\text{and}\quad a_k=\frac{(-1)^k}{2^{p-1}}\binom{2p}{p-k}\quad\text{for } 1\le k\le p\]
    as desired.
\end{remark}

Indeed, by comparing \eqref{1-cosx}, Theorem \ref{rmk5.3}, and Remark \ref{explicita_k}, we yield several new combinatorial identities.
    Also, another way to view the coefficients $a_k$'s is by using the Fourier transformation. In this case, we will have
    \begin{align*}
        a_k=\begin{cases}
            \displaystyle\frac{1}{2\pi}\int_0^{2\pi}(1-\cos x)^pdx &\text{if } k=0,\\
            \\
            \displaystyle\frac{1}{\pi}\int_0^{2\pi}(1-\cos x)^p\cos kxdx &\text{if } 1\le k\le p.
        \end{cases}
    \end{align*}
    Comparing this representation with what have been discussed also provides a relationship between the analytic integrals and the corresponding combinatorial expressions.

Let us conclude this section by computing $\mathcal{I}_{p, q}(0, \beta)$ explicitly in a considerably simple form as shown below.

\begin{corollary}
    Let $p, q\in\mathbb{N}$ such that $2\leq q\leq 2p$. Let $\beta\in\mathbb{R}$. 
    \begin{enumerate}
        \item [(a)] If $q$ is even, then
        \[\mathcal{I}_{p, q}(0, \beta)=\displaystyle \frac{\pi(-1)^{q/2}|\beta|^{q-1}}{2^{p}(q-1)!}\sum_{k=1}^p(-1)^k\binom{2p}{p-k}k^{q-1}.\]
        \item [(b)] If $q$ is odd, then
        \[\mathcal{I}_{p, q}(0, \beta)=\displaystyle\frac{(-1)^{\frac{q-1}{2}}|\beta|^{q-1}}{2^{p-1}(q-1)!}\sum_{k=1}^{p-1}(-1)^k\binom{2p}{p-k}k^{q-1}\log\frac{p}{k}.\]
    \end{enumerate}
\end{corollary}
The proof idea is the same as in Theorem \ref{mainthmcompute}, but now the coefficients are much simpler due to Remark \ref{explicita_k}, and so each expression only consists of a single sum that runs over a single index. These forms are relatively much more elementary compared to Corollary 4.4 of \cite{AL}.



\section{Frullani-type integral for difference of sines}\label{sec6}

For the ease of reference, let us make the following definitions.
Let $p, q\in\mathbb{N}$ and $\alpha, \beta\in\mathbb{R}$ such that $|\alpha|\ne|\beta|$. Denote
\[\mathcal{K}_{p, q}(\alpha, \beta):=\int_0^{\infty}h_{p, q}[\alpha, \beta](x)dx,\]
where
\[h_{p, q}[\alpha, \beta](x):=\frac{(\sin \alpha x-\sin \beta x)^p}{x^q}.\]
In this final section, following the same line of reasoning as before, we determine the convergence classification for $\mathcal{K}_{p, q}(\alpha, \beta)$. Since the arguments are nearly identical, we will provide only sketch proofs for each result. The explicit evaluation can be derived using the same method.

To begin, we need to expand
\[
\mathcal{L}(x) := (\sin \alpha x - \sin \beta x)^p
\]
as a sum of unpowered sine or cosine terms. Recall the following result before proceeding further.

\begin{lemma}\label{lem6.1}(\cite{AL}, Corollary 3.2)
	Let $n\in\mathbb{N}$. For $x\in\mathbb{R}$, we have
	\begin{align*}
		\sin^nx=\begin{cases}
			\displaystyle\frac{(-1)^{\frac{n(n+1)}{2}+1}}{2^{n-1}}\sum_{k=0}^{n-1}(-1)^k\binom{n-1}{k}\sin((n-2k)x) &\text{if $n$ is odd},\\\\
			\displaystyle\frac{(-1)^{\frac{n(n+1)}{2}}}{2^{n-1}}\sum_{k=0}^{n-1}(-1)^k\binom{n-1}{k}\cos((n-2k)x) &\text{if $n$ is even}.
		\end{cases}
	\end{align*}
\end{lemma}

\subsection{Case $q=1$}

Due to the nature of the sine function, it is in fact easier to study $\mathcal{K}_{p, q}(\alpha, \beta)$ in this setting.

\begin{theorem}
    Let $p\in\mathbb{N}$ and $\alpha, \beta\in\mathbb{R}$ such that $|\alpha|\ne|\beta|$. Then $\mathcal{K}_{p, 1}(\alpha, \beta)$ converges if and only if $p$ is odd.
\end{theorem}
\begin{proof}
    Using Lemma \ref{lem6.1}, the product-to-sum formulas, and the binomial expansion theorem, one can show that when $p$ is odd, $\mathcal{L}$ can be expressed as a sum of unpowered sine terms:
\begin{equation}\label{eq6.1}
\mathcal{L}(x) = a_0 + \sum_{k=1}^{t_1} a_k \sin d_k x,
\end{equation}
for some $t_1 \in \mathbb{N}$, where $a_0 \in \mathbb{R}$ and $a_k, d_k \in \mathbb{R} \setminus \{0\}$ are distinct for $k = 1, \ldots, t_1$.
On the other hand, if $p$ is even, $\mathcal{L}$ is a sum of unpowered cosine terms:
\begin{equation}\label{eq6.2}
\mathcal{L}(x) = b_0 + \sum_{k=1}^{t_2} b_k \cos e_k x ,
\end{equation}
with analogous assumptions on the coefficients and frequencies.

Now, in the case where $p$ is odd, substituting $x = 0$ into \eqref{eq6.1} yields $a_0 = 0$, and hence
\[
\mathcal{K}_{p, 1}(\alpha, \beta)
= \sum_{k=1}^{t_1} a_k \int_0^{\infty} \frac{\sin d_k x}{x}\, dx
= \sum_{k=1}^{t_1} \frac{\pi a_k}{2} \operatorname{sgn}(d_k)
< \infty.
\]
This shows that $\mathcal{K}_{p, 1}(\alpha, \beta)$ converges.
For the even $p$ case, one may use an Abel summation argument similar to that in Theorem \ref{thm3.1}. Noting that $b_0\ne 0$ and $\sum_{k=0}^{t_2} b_k = 0$, we can then establish the divergence of $\mathcal{K}_{p, 1}(\alpha, \beta)$.
\end{proof}

For the exact computation (which is left as an exercise), one must carefully track the coefficients $a_k$ and the signs of $d_k$.

\subsection{Cases $q\ge 2$}

Before proceeding to the result, we state a useful lemma.

\begin{lemma}\label{lem6.3}
Let $\alpha\in[-1,1)$ and $p, q\in\mathbb{N}$ with $q>p>0$.
\begin{enumerate}
    \item [(a)] If $p$ is even, then $h_{p, q}[\alpha, 1](x)$ is positive and decreasing on $(0, \pi/2]$.
    \item [(b)] If $p$ is odd, then $h_{p, q}[\alpha, 1](x)$ is negative and increasing on $(0, \pi/2]$. 
\end{enumerate} 
\end{lemma}
\begin{proof}
We begin by observing that, over $(0, \pi/2]$, the function
\[\xi(x):=\frac{\sin \alpha x-\sin x}{x}\]
is negative and increasing. This can be shown by applying the inequality \[\alpha x\cos\alpha x+\sin x\ge x\cos x+\sin\alpha x\] to $\xi'$. The proof here will be omitted.
Then we observe that $h_{p, q}[\alpha, 1](x)=\xi^{p}x^{p-q}$. So the assertion follows from the parity of $p$.
\end{proof}

\begin{theorem}
   Let $p, q\in\mathbb{N}$ and $\alpha, \beta\in\mathbb{R}$ such that $q\ge 2$ and $|\alpha|\ne|\beta|$. Then $\mathcal{K}_{p, q}(\alpha, \beta)$ exists if and only if $q\le p$.
\end{theorem}
\begin{proof}
	($\Leftarrow$) We may follow the same trick as in Theorem \ref{compar}. That is, we split the integral into two disjoint parts: \[\mathcal{K}_{p, q}(\alpha, \beta)=\int_0^1 h_{p, q}[\alpha, \beta](x)dx+\int_1^{\infty} h_{p, q}[\alpha, \beta](x)dx.\] The second one trivially converges due to the comparison test. For the first integral, we only need to observe that
	$(\sin \alpha x-\sin \beta x)^p=(\alpha-\beta)^px^p+O(x^{p+2})$,
	and then show that $\lim_{x\to 0}h_{p, q}[\alpha, \beta](x)$ exists for $2\le q\le p$.

    ($\Rightarrow$) It suffices to consider when $\alpha\in[-1, 1)$ and $\beta=1$ by the substitution rule. Our goal is to show that $\int_0^1h$ diverges. There are two cases to consider: $p$ is even and $p$ is odd. However, they appear to be similar (just that when $p$ is odd, we instead consider $|h|$). Hence let us only consider when $p$ is even. Similar to Theorem \ref{divq>=2}, by Lemma \ref{lem6.3}, we observe that
    \[\int_{\frac{1}{k+1}}^{\frac{1}{k}}h_{p, q}[\alpha, \beta](x)dx\ge \left(\sin\left(\frac{\alpha}{k}\right)-\sin\left(\frac{1}{k}\right)\right)^p\frac{k^{q-1}}{k+1}>0.\]
    So we have
    \[A_k:=\int_{\frac{1}{k+1}}^1h_{p, q}[\alpha, \beta](x)dx\ge\sum_{n=1}^k\left(\sin\left(\frac{\alpha}{n}\right)-\sin\left(\frac{1}{n}\right)\right)^p\frac{n^{q-1}}{n+1}>0.\]
    Then one can finally show (using the divergence test or the limit comparison test) that the sum on the right-hand side diverges as $k\to\infty$.
\end{proof}

\subsection{Final remark}

Of course, one might also ask whether there exists a direct relationship between the two families $\mathcal{I}_{p, q}(\alpha, \beta)$ and $\mathcal{K}_{p, q}(\alpha, \beta)$. This question, however, is not easy to answer. One possible reason is that the parameter ranges $(p, q; \alpha, \beta)$ ensuring the convergence of both families are not exactly the same. Another promising direction is to study integrals involving products of sine and cosine differences, such as
\[\int_0^{\infty}\frac{(\cos\alpha x-\cos \beta x)^p(\sin\gamma x-\sin\delta x)^r}{x^q}dx.\]
Understanding this may reveal deeper connections between the two frameworks. 
These questions are left for future investigation.

\section*{Acknowledgements}

We thank Weerapat Li for revising the introduction of this paper. We also thank the referees for their constructive comments, which significantly improved the manuscript. The  first author is funded by Kamnoetvidya Science Academy, while the second author is funded by the Philip K. H. Wong Foundations Scholarship at the University of Hong Kong.

\section*{Conflict of interest}

The authors declare no conflict of interest.

\end{document}